\documentclass[11pt]{amsart}
\usepackage{preamble}

\title[Artin's Conjecture in Number Fields and For Matrices]{Artin's Primitive Root Conjecture in Number Fields and For Matrices}

\author{Noam Kimmel}
\address{N. Kimmel: Raymond and Beverly Sackler School of Mathematical Sciences, Tel Aviv University, Tel Aviv 69978, Israel.}
\email{\href{mailto:noamkimmel@mail.tau.ac.il}{noamkimmel@mail.tau.ac.il}}
\thanks{This research was supported by the European Research Council (ERC) under the European Union's  Horizon 2020 research and innovation program  (Grant agreement No.    786758).}

\begin{document}

\maketitle 

\begin{abstract}
In 1927, E. Artin conjectured that all non-square integers $a\neq -1$ are a primitive root of $\mathbb{F}_p$ for infinitely many primes $p$.
In 1967, Hooley showed that this conjecture follows from the Generalized Riemann Hypothesis (GRH).
In this paper we consider variants of the primitive root conjecture for number fields and for matrices.
All results are conditional on GRH.

For an algebraic number field $K$ and some element $\alpha \in K$, we examine the order of $\alpha$ modulo various rational primes $p$.
We extend previous results of Roskam which only worked for quadratic extensions $K/\mathbb{Q}$ to more general field extensions of higher degree.
Specifically, under some constraints on the Galois group of $K/\mathbb{Q}$ and on the element $\alpha\in K$, we show that $\alpha$ is of almost maximal order mod $p$ for almost all rational primes $p$ which factor into primes of degree 2 in $K$.

We also consider Artin's primitive root conjecture for matrices.
Given a matrix $A\in\text{GL}_n(\mathbb{Q})$, we examine the order of $A\bmod p$ in $\text{GL}_n(\mathbb{F}_p)$ for various primes $p$, which turns out to be equivalent to the number field setting.
\end{abstract}

\setcounter{tocdepth}{1}
\tableofcontents

\newpage
\section{Introduction}
\subsection{Artin's primitive root conjecture}

It was conjectured by Artin that every integer $a\in\ZZ$ such that $a\neq-1$ and $a$ is not a perfect square, is a primitive root of $\FF_p$ for infinitely many primes $p$.
In 1967 Hooley showed that this follows from the Generalized Riemann Hypothesis (GRH) \cite{hooley}. 
In fact, Hooley showed that assuming GRH, the number of primes $p\leq X$ for which $a$ is a primitive root is asymptotic to $c\li{X}$ where $c$ is a constant depending on $a$.

Instead of showing that $a$ is a primitive root for a positive proportion of the primes, it is sometimes useful in applications to show that $a$ has almost maximal order in $\FF_p^\times$ for almost all primes $p$.
Hooley's method in \cite{hooley} can be adapted to tackle this version as well.
This was done by Erd\"os and Murty \cite[Theorem 4]{erdos1999order} which gives
\begin{theorem}\label{thm:1.1}
Let $a\in\ZZ$, $a\neq \pm 1$, and let $f(x)$ be some function tending to infinity with $x$.
Then assuming GRH:
$$
\#\SET{p\leq X \;\middle| \; \ord{p}{a} \leq \frac{p-1}{f(X)}} = o(\li{X}).
$$
\end{theorem}

Various generalizations of Artin's conjecture to other algebraic structures have also been studied (see for example the survey \cite{moree2012artin}).
In this paper, we examine two generalizations of \autoref{thm:1.1}.
In \autoref{sec:2} we examine the analogue of \autoref{thm:1.1} in the case of number fields, and in \autoref{sec:3} we examine a generalization to matrices, which we show to be related to the number field case.

\subsection{Extension to number fields}
\autoref{thm:1.1} can be extended to the case of a number field $K/\QQ$ with some element $\alpha\in K$ in two natural ways.
The first is in terms of the primes in $K$.
That is, one can ask whether $\alpha$ is of almost maximal order in $\left(\mathcal{O}_K/\mathfrak{P}\right)^\times$ for almost all primes $\mathfrak{P}$ in $K$ of norm less than $X$.
In this case, the question reduces to linear primes, since almost all primes in $K$ up to norm $X$ are linear.
It turns out that Hooley's original proof for the rational case works in this case as well, with only minor cosmetic changes.
This was done in \cite[Theorem 3.1]{olli}.
We summarize this result in the following proposition.
\begin{proposition}\label{prop:1.2}
Let $f(x)$ be some function tending to infinity with $x$.
Let $K$ be some number field and let $\alpha$ be an element in $K$ which is not a root of unity.
Assuming GRH, we have:
\begin{multline*}
\#\SET{\mathfrak{P} \text{ prime in }K \;\middle|\; \mathcal{N}_{K/\QQ}(\mathfrak{P})\leq X,\; \ord{\mathfrak{P}}{\alpha} \leq \frac{\mathcal{N}_{K/\QQ}(\mathfrak{P}) - 1}{f(X)}} \\
= o(\li{X}).
\end{multline*}
\end{proposition}
As a consequence, we also have
\begin{corollary}\label{cor:1.3}
With the notations and assumptions of \autoref{prop:1.2},
$$
\#\SET{p \leq X \;\middle|\; p\text{ splits completely in }K \text{ and }\ord{p}{\alpha} \leq \frac{p-1}{f(X)}}
= o(\li{x}).
$$
\end{corollary}

The second way \autoref{thm:1.1} can be generalized to the number field case is in terms of rational primes.
One wonders whether $\alpha$ has almost maximal possible order in $\left(\mathcal{O}_K/p\mathcal{O}_K\right)^\times$ for almost all rational primes $p\leq X$.
There is a positive proportion of rational primes $p$ which do not split completely in $K$.
Thus, one is led to consider cases where the factorization of $p$ in $K$ contains non-linear primes.

For these non-linear primes, Hooley's method doesn't work.
There are two main problems for such primes.
First, unlike the linear case, such a prime $\mathfrak{P}$ lying above $p\leq X$ will have norm $\mathcal{N}_{K/\QQ}(\mathfrak{P}) \geq p^2$ which can be significantly larger than $X$.
The second problem has to do with the error terms when applying Chebotarev’s density theorem.
Even when assuming GRH, the error terms become too large when applying Hooley's method for such primes.

The first works to deal with non-linear primes was due to Roskam in \cite{roskam2000quadratic}, \cite{roskam2002artin}.
In these works, Roskam showed that under GRH, if $K=\QQ(\alpha)$ is a quadratic field, and if we assume that $\alpha$ and its conjugate are multiplicatively independent, then $\alpha$ has maximal possible order in $\left(\mathcal{O}_K/p\mathcal{O}_K\right)^\times$ for a positive proportion of rational primes $p$, even when only considering non-split primes $p$. 

In the quadratic case, the problem of large norms can be dealt with by noting that for a prime $\mathfrak{P}$ of degree two, we have 
$$
\left| \left(\mathcal{O}_K/\mathfrak{P}\right)^\times\right|
=
\mathcal{N}_{K/\QQ}(\mathfrak{P})  - 1 = 
p^2 - 1 = (p-1)(p+1).
$$
The fact that the size of the multiplicative group of $\mathcal{O}_K/\mathfrak{P}$ factors into linear terms allows part of Hooley's method to be extended to such primes $\mathfrak{P}$.
As for the error terms from Chebotarev’s density theorem, Roskam was able to overcome this problem by applying the theorem to special sub-fields which still contained the required information, but gave sufficiently small error terms.

However, Roskam's method does not extend to field extensions $K/\QQ$ of degree larger than 2.
For such extensions, even if we restrict our attention to rational primes $p$ which factor into primes of degree one or two in $K$, the problem with the error terms from Chebotarev’s density theorem is still present.
In this paper, we consider special field extensions $K/\QQ$ with $[K:\QQ] > 2$ for which we are able to handle such rational primes.

\subsection{Main Results}
Let $K/\QQ$ be a Galois extension of $\QQ$, and let $\alpha$ be an element in $K$.
We will denote $G = \mathrm{Gal}(K/\QQ)$.
Let $C\subset G$ be some conjugacy class of elements of order 2.
Denote 
$$
\PP_C(X) = \SET{p \leq X \;\middle|\; \left(\frac{K/\QQ}{p}\right) = C}
$$
where $\left(\frac{K/\QQ}{\cdot}\right)$ denotes the Artin symbol.
\begin{theorem}\label{thm:1.4}
With the notations above, assume that there is a normal subgroup $N\triangleleft G$ and denote $\phi : G \rightarrow G/N$ the quotient homomorphism.
Let $C\subset G$ be a conjugacy class of elements of order 2 such that $C\not\subset N$ and $\phi(C)$ is contained in the center of $G/N$.
Assume also that for any $c\in C$,
\begin{equation}\label{eq-assump_n/cn}
\left.\prod_{n\in N}n(\alpha) \middle/ c\left(\prod_{n\in N}n(\alpha)\right)\right.
\end{equation}
is not a root of unity.
Then GRH implies the following:

If $\left|\mathcal{N}_{\QQ/K}(\alpha)\right| \neq 1$ then almost all primes $p\in \PP_C(X)$ are such that $\alpha$ is of almost maximal order in $\left(\mathcal{O}_{K} / p\mathcal{O}_{K}\right)^\times$.
In other words, if $f(x)$ is a function tending to infinity with $x$, then 
$$
\#\SET{p\in \PP_C(X) \;\middle| \;
\ord{p}{\alpha} \leq 
\frac{p^2 - 1}{f(X)}} = o(\li{X}).
$$

If $\left|\mathcal{N}_{\QQ/K}(\alpha)\right| = 1$ then for almost all primes $p\in \PP_C(X)$, $\alpha$ has order almost $p+1$ in $\left(\mathcal{O}_{K} / p\mathcal{O}_{K}\right)^\times$.
That is, if $f(x)$ is a function tending to infinity with $x$, then 
$$
\#\SET{p\in \PP_C(X) \;\middle| \;
\ord{p}{\alpha} \leq 
\frac{p+1}{f(X)}} = o(\li{X}).
$$

\end{theorem} 

\begin{remark}
In the case $\mathcal{N}_{\QQ/K}(\alpha) = 1$, having an order of $p+1$ is the largest possible order, since in this case the image of $\alpha$ in the quotient, and its powers, live inside a subgroup of exponent $p+1$ of the group $\left(\mathcal{O}_{K} / p\mathcal{O}_{K}\right)^\times$ (the subgroup of elements of norm 1).
Similarly, if $\mathcal{N}_{\QQ/K}(\alpha) = -1$ then $\ord{p}{\alpha}$ is at most $2(p+1)$ since $\alpha^2$ has norm $1$.
\end{remark}

As an example of \autoref{thm:1.4}, we consider the case where $G$ is abelian.
In this case, the conditions of \autoref{thm:1.4} regarding $G$ are satisfied if one chooses $N = \SET{1}$.
The conjugacy class $C$ must consist of a single element, and this element will trivially be in the center of $G/N\cong G$.
And so we get the following corollary.
\begin{corollary}\label{cor:1.5}
Let $K$ be a Galois extension of $\QQ$ with abelian Galois group and let $\alpha\in K$. 
Assume that for any $\sigma\in \mathrm{Gal}(K/\QQ)$ of order 2, $\frac{\alpha}{\sigma(\alpha)}$ is not a root of unity.
Denote by $\PP_{[2]}(X)$ the set of rational primes $p\leq X$ which factor into primes with inertia degree 2 in $K$.
Let $f(x)$ be some function tending to infinity with $x$.
Then, GRH implies the following:

{\setlength{\parindent}{0cm}
If $|\mathcal{N}_{K/\QQ}(\alpha)|\neq 1$ then
$$
\#\SET{p\in \PP_{[2]}(X) \;\middle| \; \ord{p}{\alpha}\leq \frac{p^2 - 1}{f(X)}} = o(\li{X}).
$$
If $\left|\mathcal{N}_{\QQ/K}(\alpha)\right| = 1$ then
$$
\#\SET{p\in \PP_{[2]}(X) \;\middle| \; \ord{p}{\alpha}\leq \frac{p+1}{f(X)}} = o(\li{X}).
$$
}
\end{corollary}

As a consequence, we also get the following.
\begin{corollary}\label{cor:1.6}
Let $K$ be a multiquadratic Galois extension of $\QQ$,
so that $\gal{K/\QQ}\cong \left(\ZZ/2\ZZ\right)^m$ for some $m\geq 1$, and let $\alpha\in K$.
Assume that $|\mathcal{N}_{K/\QQ}(\alpha)|\neq 1$ and that for any $\sigma\in\gal{K/\QQ}$ of order 2 (i.e. $\sigma\neq \mathrm{id}$), we have that $\frac{\alpha}{\sigma(\alpha)}$ is not a root of unity.
Then assuming GRH, for almost all primes $p\leq X$, $\ord{p}{\alpha}$ is of almost maximal order in $\left(\mathcal{O}_K/p\mathcal{O}_K\right)^\times$.
That is, if $f(x)$ is a function tending to infinity with $x$, then
$$
\#\SET{p\leq X \;\middle| \;\ord{p}{\alpha}\leq \frac{\exp{\left(\mathcal{O}_K/p\mathcal{O}_K\right)^\times}}{f(X)}} = o(\li{X}).
$$
\end{corollary}
\begin{proof}
In the multiquadratic field $K$ there are 2 ways in which a rational prime $p$ can split. 
Either $p$ splits completely in $K$, or it factors into primes with inertia degree 2.
For the first case we can use \autoref{cor:1.3}, and for the second case we use \autoref{cor:1.5}.
In either case, we find that there are only a negligible amount of rational primes $p\leq X$ for which 
$$
\ord{p}{\alpha}\leq \frac{\exp{\left(\mathcal{O}_K/p\mathcal{O}_K\right)^\times}}{f(X)}.
$$
\end{proof}

\begin{example}
Let $n_1,...,n_m\in\NN$, $n_i > 1$, be such that $n_1 n_2 ... n_m$ is square-free, and let $c_0,c_1,...,c_m\in\QQ\setminus\SET{0}$.
Consider the element $\alpha = c_0 + c_1\sqrt{n_1} + ... + c_m\sqrt{n_m}$ in $K = \QQ(\sqrt{n_1}, \sqrt{n_2} , ... , \sqrt{n_m})$.
For $\sigma\in\gal{K/\QQ}$, $\sigma\neq \mathrm{id}$, we have that $\frac{\alpha}{\sigma(\alpha)}\neq \pm 1$.
This follows from a theorem of Besicovitch \cite{besicovitch1940linear} which asserts that $\SET{1, \sqrt{n_1},...,\sqrt{n_m}}$ are linearly independent over $\QQ$.
If we further assume that $|\mathcal{N}_{K/\QQ}(\alpha)|\neq 1$, then from \autoref{cor:1.6} we get that assuming GRH, $\alpha$ has almost maximal order mod $p$ for almost all rational primes $p$.

So for example, let $\alpha = 1 + \sqrt{3} + \sqrt{5} \in \QQ(\sqrt{3},\sqrt{5})$.
Then $\mathcal{N}(\alpha) = -11$, and so if we assume GRH then $\alpha$ has almost maximal order in $\left(\mathcal{O}_K/p\mathcal{O}_K\right)^\times$ for almost all rational primes $p$.
\end{example}

\autoref{thm:1.4} can also be used to tackle cases where the Galois group is not abelian.
As an example, we consider the case where $G = D_m$, the dihedral group of order $2m$.
\begin{example}
\autoref{cor:1.5} also holds when $\mathrm{Gal}\left(K/\QQ\right) = D_m$.
\end{example}
\begin{proof}
We use the same notations as in \autoref{cor:1.5}.
We write 
$$
D_m = \BRAK{a,b \; \middle| \; a^m = 1,\, b^2 = 1,\, ab=ba^{-1}}.
$$
If $m$ is odd, then all elements of order 2 are of the form $ba^i$ for some $i$.
In this case we can apply \autoref{thm:1.4} with $N = \BRAK{a}$ and any conjugacy class $C \subset b\BRAK{a}$.

If $m=2k$ is even, then $a^k$ is also an element of order 2.
The element $a^k$ is in the center of $D_m$.
Thus, appealing to \autoref{thm:1.4} with the conjugacy class $C' = \SET{a^k}$ and the normal subgroup $N' = \SET{1}$ handles primes in $\PP_{C'}(X)$.

Overall, this shows that 
\begin{align*}
|\mathcal{N}_{K/\QQ}(\alpha)| \neq 1  & \text{ implies}&
\#\SET{p\in \PP_{[2]}(X) \;\middle| \; \ord{p}{\alpha} \leq \frac{p^2 - 1}{f(X)}} = o(\li{X})\\
\left|\mathcal{N}_{K/\QQ}(\alpha)\right| = 1   & \text{ implies} &
\#\SET{p\in \PP_{[2]}(X) \;\middle| \; \ord{p}{\alpha} \leq \frac{p+1}{f(X)}} = o(\li{X})
\end{align*}
as required.
\end{proof}

In \autoref{sec:3} we prove \autoref{cor:3.2} which is an analogue of \autoref{thm:1.4} for matrices.
For a rational matrix $A\in\mathrm{GL}_n(\QQ)$ we can consider the order of $A\bmod p$ in $\mathrm{GL}_n(\FF_p)$ for various rational primes $p$.
One can then ask if $A$ has almost maximal order for almost all elements in some subset of the rational primes. 
Even though the formulation of this variation does not mention algebraic numbers explicitly, it turns out that the results regarding Artin's primitive root conjecture for number fields can be used to tackle this question as well.

The case of $2\times2$ matrices in $\mathrm{SL}_2(\ZZ)$ was handled by Kurlberg in \cite[Theorem 2]{PärKurlberg2003} using ideas from Roskam's work \cite{roskam2002artin}. 
\autoref{cor:3.2} is an extension of this result to matrices of dimension larger than 2.

We also note that the matrix variation of Artin's primitive root conjecture has applications in problems arising in quantum chaos (see for example the works of Kurlberg and Rudnick \cite{kurlberg2000hecke} \cite{kurlberg2001quantum} and of Kelmer \cite{kelmer2010arithmetic} regarding quantum cat maps).

\subsection{Notations}
\begin{itemize}[leftmargin=*]
    \item For a number field $K$ we denote by $\mathcal{O}_K$ its ring of integers.
    \item For $\alpha = \alpha_0/\alpha_1$ ($\alpha_0,\alpha_1 \in \mathcal{O}_K$) co-prime to $p$ in $K$ we denote $\ord{p}{\alpha}$ the order of $\alpha$ in $\left(\mathcal{O}_{K} / p\mathcal{O}_K\right)^\times$.
    \item If $\mathfrak{P}$ is a prime in $\mathcal{O}_K$ above a prime $\mathfrak{p}$ in $\mathcal{O}_L$, we denote by $\left(\frac{K/L}{\mathfrak{P}}\right)\in \gal{K/L}$ its Artin symbol.
    We also denote $\left(\frac{K/L}{\mathfrak{p}}\right)$ the conjugacy class of $\gal{K/L}$ of all Artin symbols of primes above $\mathfrak{p}$.
    \item For a group $G$ we denote the exponent of $G$ by $\exp{G}$.
    \item Throughout, the letters $p,q$ will denote rational prime numbers.
    \item We denote by $\zeta_n$ a primitive $n\th$ root of unity.
    \item $M_n(R)$ denotes the set of all $n\times n$ matrices with entries from the ring $R$, and $\mathrm{GL}_n(R)\subset M_n(R)$ denotes the group of all invertible matrices.
    \item For a matrix $A\in M_n(\QQ)$ with $p\nmid \det A$, we denote by $\ord{p}{A}$ the order of $A\mod p$ in $ \mathrm{GL}_n(\FF_p)$.
    \item We use the standard Big O notation. 
    That is, $f(x) = \BigO{g(x)}$ if there are $C>0, N>0$ such that for all $x > N$: $|f(x)| \leq C g(x)$.
    \item We use the notation $f(x) \ll g(x)$ to indicate $f(x) = \BigO{g(x)}$.
\end{itemize}

\section{Preliminary results}

In this section we gather several results which we will
need in the proof of \autoref{thm:1.4}.
We begin by stating the effective version of Chebotarev's density theorem.
\begin{theorem}[GRH conditional Chebotarev's density theorem]\label{thm:2.1}
Let $L/K$ be a Galois extension with $\gal{L/K} = G$, and let $C\subset G$ be some conjugacy class.
Then assuming GRH:
\begin{multline*}
\#\SET{\mathfrak{P} \text{ prime in }K \;\middle|\; N(\mathfrak{P}) \leq X, \; \left(\frac{L/K}{\mathfrak{P}}\right) = C} = \\
\frac{|C|}{|G|}\li{X} + \BigO{|C|[K:\QQ]\sqrt{X}\log X + \frac{|C|}{|G|} \sqrt{X}\log |\Delta|}
\end{multline*}
where $\Delta$ is the discriminant of $L$ over $\QQ$.
\end{theorem}

\begin{proof}
\cite[Theorem 1]{lagarias1977effective}.
\end{proof}

The error term in the theorem above involves the discriminant of $L$ over $\QQ$.
In our case, the field $L$ will have the form $L = K\left(\zeta_q, \alpha_1^{1/q}, \alpha_2^{1/q},...,\alpha_m^{1/q}\right)$ for some elements $\alpha_1,\alpha_2,...,\alpha_m\in K$.
And so, we state a result which gives a bound on the discriminant of such fields.
Specifically, our next lemma will be used to show that the term $\frac{|C|}{|G|} \sqrt{X}\log |\Delta|$ in \autoref{thm:2.1} is not going to be too large.
More specifically, we will show that in the cases we consider, we have
\begin{equation}\label{eq:2.1}
\frac{|C|}{|G|} \sqrt{X}\log |\Delta| \ll \sqrt{X}\log X
\end{equation}
with implied constants depending on $C,K,\alpha_1,...,\alpha_m$ but not on $q$.

\begin{lemma}\label{lem:2.2}
Let $K$ be a Galois number field over $\QQ$.
Let $\alpha_1,\alpha_2,...,\alpha_m\in K$, and consider 
$$
L = K\left(\zeta_q, \alpha_1^{1/q},\alpha_2^{1/q},...,\alpha_m^{1/q}\right).
$$
Denote by $\Delta$ the discriminant of $L$ over $\QQ$.
Then we have that
$$
\log |\Delta| \ll [L:\QQ]\log q
$$
where the implied constants depend on $K,\alpha_1,...,\alpha_m$.
\end{lemma}
\begin{remark}
The bound \eqref{eq:2.1} then follows from \autoref{lem:2.2} for $q\leq X$, since in this case $\log q \leq \log X$.
\end{remark}
\begin{proof}
This lemma appears in various forms in the
literature, e.g. \cite[Theorem 23]{MR3994150}.
We present here a proof for completeness.

We prove this by induction on $m$.
Throughout the proof, if $K/F$ is a field extension, we denote by $\Delta_{K/F}$ the relative discriminant of $K$ over $F$.
We use the known relationship for the relative discriminants in a tower fields $K/L/F$:
\begin{equation}\label{eq:disc_tower}
\Delta_{K/F} = \mathcal{N}_{L/F}\left(\Delta_{K/L}\right)
\Delta_{L/F}^{[K:L]}.
\end{equation}
We begin with the case $m=0$.
In this case we look at $K(\zeta_q) / K / \QQ$.
We have that 
$$
[K(\zeta_q):K] = \bigslant{[K(\zeta_q):\QQ]}{[K:\QQ]} \gg q
$$
with implied constants depending on $K$.
We also have:
$$
\left|
\mathcal{N}_{K/\QQ}\left(\Delta_{K(\zeta_q)/K}\right)
\right|
\leq
\left|
\mathcal{N}_{K/\QQ}\left(
\prod_{0\leq i < j < q-1}(\zeta_q^i-\zeta_q^j)^2 \right)
\right|
=
\left(q^{q-2}\right)^{[K:\QQ]}.
$$
From \eqref{eq:disc_tower} we then get 
$$
|\Delta_{K(\zeta_q)/\QQ}| \leq q^{(q-2)[K:\QQ]}\left|
\Delta_{K/\QQ}\right|^{q-1}.
$$
And so we get
$$
\log |\Delta_{K(\zeta_q)/\QQ}| \ll [K(\zeta_q):\QQ]\log q
$$
with implied constants that depend on $K$.

Denote now
$$
K_{t} = K\left(\zeta_q,\alpha_1^{1/q},...,\alpha_t^{1/q}\right).
$$
We assume $|\Delta_{K_t/\QQ}|\ll [K_t:\QQ]\log q$ and we wish to show $|\Delta_{K_{t+1}/\QQ}|\ll [K_{t+1}:\QQ]\log q$.
If $\alpha_{t+1}^{1/q}\in K_t$ then $K_{t+1} = K_t$ and the claim is trivially satisfied.
Assume then that $\alpha_{t+1}^{1/q}\not \in K_t$.
Consider the tower of extensions $K_{t+1}/K_t/\QQ$.
We calculate
$$
\Delta_{K_{t+1}/K_t} = 
\prod_{0\leq i < j < q} \left(\alpha_{t+1}^{1/q}\zeta_q^i - \alpha_{t+1}^{1/q}\zeta_q^j\right)^2
= (-1)^{\frac{q-1}{2}}\alpha_{t+1}^{q-1}q^{q-2}.
$$
This gives
$$
\left|\mathcal{N}_{K_{t}/\QQ}(\Delta_{K_{t+1}/K_t})\right|
=
\left| \mathcal{N}_{K/\QQ}(\alpha_{t+1})\right|^{(q-1)[K_{t}:K]}
q^{(q-2)[K_{t}:\QQ]}.
$$
We now apply \eqref{eq:disc_tower}.
For $\Delta_{K_t/\QQ}$ we use the induction hypothesis.
This gives us
\begin{multline*}
\log |\Delta_{K_{t+1}/\QQ}| \ll 
\log \left|\mathcal{N}_{K/\QQ}(\alpha_{t+1})\right|(q-1)[K_{t}:K]
+\\
(q-2)[K_{t}:\QQ] \log q 
+ 
[K_{t+1}:K_{t}][K_{t}:\QQ]\log q.
\end{multline*}
Noting that $[K_{t+1}:\QQ] = q[K_{t}:\QQ]$ we then get
$$
\log |\Delta_{K_{t+1}/\QQ}| \ll [K_{t+1}:\QQ]\log q
$$
with implied constants depending on $K,\alpha_1,...,\alpha_{t+1}$ but not on $q$.

\end{proof}

Lastly, we will also require the Brun–Titchmarsh theorem.
\begin{theorem}[Brun–Titchmarsh]\label{thm:2.3}
For $d<X$, denote by $\pi(X,d,a)$ the number of primes $p\leq X$ with $p\equiv a \bmod{d}$.
Then $$\pi(X,d,a) < \frac{2X}{\varphi(d)\log(X/d)}.$$
\end{theorem}

This is \cite[Theorem 2]{montgomery1973large}.

\section{Proof of Main Results for Number Fields}\label{sec:2}
In this section we prove \autoref{thm:1.4}.
Throughout this section we use the notation of \autoref{thm:1.4}.
We wish to show that if $|\mathcal{N}_{K/\QQ}(\alpha)|\neq 1$ then the number of primes $p\in\PP_C(X)$ for which $\ord{p}{\alpha}$ is smaller than $\frac{p^2 - 1}{f(X)}$ is $o(\li{X})$, and that if $\left|\mathcal{N}_{\QQ/K}(\alpha)\right| = 1$ then the number of primes $p\in\PP_C(X)$ for which $\ord{p}{\alpha}$ is smaller than $\frac{p + 1}{f(X)}$ is $o(\li{X})$.

For $p\in\PP_C(X)$, we have that 
$$
\exp{\left(\mathcal{O}_K/p\mathcal{O}_K\right)^\times} = p^2 - 1.
$$
The only reason $\ord{p}{\alpha}$ might be smaller than $p^2 - 1$ is if there is some prime $q$ such that $q\mid p^2 - 1$ and $\alpha$ is a $q\th$ power mod $p$.
If such a prime $q$ exists, we say that $p$ has an obstruction at $q$.
Since $p^2 - 1 = (p-1)(p+1)$, if $p$ has an obstruction at $q$ then either $q\mid p-1$ or $q\mid p+1$.
If $q\mid p-1$ we call this an obstruction of type I, and if $q\mid p+1$ we call this an obstruction of type II.
And so, before we proceed to the proof of \autoref{thm:1.4}, we first prove a lemma bounding the number of primes $p\in \PP_C(X)$ having obstructions at some sufficiently large prime $q$ with $q < \frac{\sqrt{X}}{\log^2 X}$.
The main novelty of the paper is in providing sufficiently strong bounds for the number of such $p$'s with a type II obstruction. 
The rest of the proof broadly follows Hooley's original ideas.
We will adapt some of Hooley's arguments in order to get a bound for the number of primes $p\in  \PP_C(X)$ having obstructions at some $q$ with $q >  \frac{\sqrt{X}}{\log^2 X}$.
We will then show that obstructions at very small primes also eliminate only a negligible amount of primes $p\in\PP_C(X)$. 

We will denote by $B_\alpha(X;q)$ the number of primes $p\in \PP_C(X)$ having an obstruction at $q$.
We will also denote $B_\alpha^I(X;q), B_\alpha^{II}(X;q)$ the number of primes $p\in \PP_C(X)$ having an obstruction at $q$ of type I, II respectively.
We begin by proving the following lemma.
\begin{lemma}\label{lem:2.4}
Let $g(x)$ be a function tending to infinity with $x$.
Assuming GRH, we have:
$$
\sum_{g(X) \leq q \leq \frac{\sqrt{X}}{\log^2 X}}
B^{II}_{\alpha}(X;q)
= o\left(\li{X}\right).
$$
If we further assume $|\mathcal{N}_{K/\QQ}(\alpha)|\neq 1$, then we have:
$$
\sum_{g(X) \leq q \leq \frac{\sqrt{X}}{\log^2 X}}
B_{\alpha}(X;q)
= o\left(\li{X}\right).
$$
\end{lemma}
\begin{proof}
Let $p\in\PP_C(X)$ be such that $p$ has an obstruction at $q$.
We assume that $p,q$ are both large enough so that none divides the discriminant of $K$.
Let $\mathfrak{P}_1 \mathfrak{P}_2 \cdot ...\cdot \mathfrak{P}_m$ be the prime factorization of $p\mathcal{O}_K$ in $K$ (where $m = |G|/2$).
The fact that $p$ has an obstruction at $q$ implies that $q\mid p^2-1$, and that there exists some $\beta \in\mathcal{O}_K$ such that 
\begin{equation}\label{eq:2.2}
\alpha = \beta^q \bmod{\mathfrak{P}_i},\quad i=1,...,m.
\end{equation}

We begin by first considering the case where $|\mathcal{N}_{K/\QQ}(\alpha)|\neq 1$ and $p$ has an obstruction of type I at $q$.
If $\tau\in G$ takes $\mathfrak{P}_i$ to $\mathfrak{P}_1$, then by applying $\tau$ to equation $i$ in \eqref{eq:2.2} we get
\begin{equation*}
\tau(\alpha) = \tau(\beta)^q \bmod{\mathfrak{P}_1},\quad \tau\in G.
\end{equation*}
By multiplying these equations, we get
\begin{equation*}
\mathcal{N}_{K/\QQ}(\alpha) = \mathcal{N}_{K/\QQ}(\beta)^q \bmod{\mathfrak{P}_1}.
\end{equation*}
Since both $\mathcal{N}_{K/\QQ}(\alpha)$ and $\mathcal{N}_{K/\QQ}(\beta)$ are rational, it follows that the last equation holds mod $p\ZZ = \mathfrak{P}_1 \cap \QQ$. 
That is,
\begin{equation}\label{eq-nalpha_nbeta}
\mathcal{N}_{K/\QQ}(\alpha) = \mathcal{N}_{K/\QQ}(\beta)^q \bmod{p}.
\end{equation}

Consider the field
$$
L_q^I = \QQ\left(\zeta_q, \mathcal{N}_{K/\QQ}(\alpha)^{1/q}\right).
$$
We show that for $p\nmid \mathcal{N}_{K/\QQ}(\alpha)$, $p$ splits completely in $L_q^I$.
Indeed, $p$ has an obstruction of type I at $q$, so that $p\equiv 1 \bmod q$.
This implies that $p$ splits completely in $\QQ(\zeta_q)$.
Consider now how $p$ factors in $\QQ\left(\mathcal{N}_{K/\QQ}(\alpha)^{1/q}\right)$.
The primes appearing in the discriminant of $\QQ\left(\mathcal{N}_{K/\QQ}(\alpha)^{1/q}\right)$ divide either $q$ or $\mathcal{N}_{K/\QQ}(\alpha)$.
The condition \eqref{eq-nalpha_nbeta} together with $p\equiv 1\bmod q$ imply that the equation $T^q \equiv \mathcal{N}_{K/\QQ}(\alpha) \bmod p$ has $q$ solutions.
It follows by a criteria of Dedekind that $p$ splits in $\QQ\left(\mathcal{N}_{K/\QQ}(\alpha)^{1/q}\right)$ (since we assumed $p\nmid (\mathcal{N}_{K/\QQ}(\alpha)$, and we know that $p\nmid q$).
Thus, $p$ splits in the compositum $L_q^I$.

By assumption, we have that $|\mathcal{N}_{K/\QQ}(\alpha)|\neq 1$.
Thus, for large enough $q$ (such that $\mathcal{N}_{K/\QQ}(\alpha)^{1/q} \not \in \QQ$), we have $[L_q^I:\QQ] = q(q-1)$.

And so, from \autoref{thm:2.1} and \autoref{lem:2.2} we get that, assuming GRH,
$$
B_\alpha^I(X;q) \leq \frac{\li{X}}{q(q-1)} + \BigO{\sqrt{X}\log X}.
$$
Summing this for all $g(X) \leq q \leq \frac{\sqrt{X}}{\log^2 X}$ we get
$$
\sum_{g(X) \leq q \leq \frac{\sqrt{X}}{\log^2 X}}
B^I_{\alpha}(X;q) \leq 
\frac{\li{X}}{g(X)} + \BigO{\frac{X}{\log^2 X}} = o(\li{X}).
$$

We now consider the case where $p$ has an obstruction of type II at $q$.
Denote by $M$ the sub-field of $K$ fixed by the normal subgroup $N$.
Thus, $M$ is a Galois extension of $\QQ$.
Since $\left(\frac{K/\QQ}{p}\right) = C$ and since $C\not\subset N$, the prime factorization of $p\mathcal{O}_M$ in $M$ has the form $\mathfrak{p}_1 \mathfrak{p}_2 \cdot ... \cdot \mathfrak{p}_{m'}$, where the $\mathfrak{p}_i$ have inertia degree 2 over $p$ (and $m' = |G|/2|N|$).
\begin{center}
\begin{tikzpicture}
    \node (Q1) at (-1,0) {$\mathbb{Q}$};
    \node (Q2) at (-1,1) {$M$};
    \node (Q3) at (-1,2) {$K$};
    
    \node (Q4) at (1,0) {$(p)$};
    \node (Q5) at (1,1) {$\mathfrak{p}_1\mathfrak{p}_2\cdot ...\cdot \mathfrak{p}_{m'}$};
    \node (Q6) at (1,2) {$\mathfrak{P}_1\mathfrak{P}_2\cdot ...\cdot \mathfrak{P}_m$};

    \draw (Q1)--(Q2);
    \draw (Q2)--(Q3);
    \draw (Q4)--(Q5);
    \draw (Q5)--(Q6);
\end{tikzpicture}
\end{center}
From equations \eqref{eq:2.2}, employing a similar strategy as we did before, we get that
\begin{equation}\label{eq:2.3}
\mathcal{N}_{K/M}(\alpha) = \left(\mathcal{N}_{K/M}(\beta)\right)^q \bmod{\mathfrak{p}_i},\quad i=1,...,m'.
\end{equation}
To ease notation, denote $\alpha_M = \mathcal{N}_{K/M}(\alpha)$.
Since $C$ is a conjugacy class in $G$, and the image of $C$ in $G/N$ (which we denoted $\phi(C)$) is in the center of $G/N$, it follows that $\phi(C)$ corresponds to a single element $\sigma \in \gal{M/\QQ} \cong G/N$.
It can also be seen that $\left(\frac{M/\QQ}{\mathfrak{p}_i}\right) = \sigma$ for all $i=1,2,...,m'$.
We now consider the field
$$
L_q^{II} = M\left(\zeta_q\right)\left(
 \SET{\left(\frac{\tau(\alpha_M)}
 {\sigma(\tau(\alpha_M))}\right)^{1/q}}_{\tau \in \gal{M/\QQ} / \BRAK{\sigma}}
\right).
$$
The notation $\tau \in \gal{M/\QQ} / \BRAK{\sigma}$ means that from each pair of elements $\tau,\sigma \tau \in \gal{M/\QQ}$, we only pick one representative.
We first show that $L_q^{II}$ is Galois over $\QQ$.
For this, consider the polynomial
$$
F(T) = \prod_{\tau \in \gal{M/\QQ}}
\left(T^q - \frac{\tau(\alpha_M)}{\sigma(\tau(\alpha_M))}\right).
$$
Using the fact that $\sigma$ is in the center of $\gal{M/\QQ}$, we see that this polynomial is invariant under $\gal{M/\QQ}$.
Thus, we have that $F(T)\in\QQ[T]$.
The fact that $L_q^{II}$ is Galois over $\QQ$ then follows since $L_q^{II}$ is the splitting field of $F(T)$.

We now show that there is a singe element $\Tilde{\rho}\in\gal{L_q^{II}/\QQ}$ which is the Artin symbol of all $p\in\PP_C(X)$ having a type II obstruction at $q$.
To see this, let $\rho$ be the Artin symbol $\left(\frac{L_q^{II}/\QQ}{\mathfrak{P}}\right)$ of some prime $\mathfrak{P}$ in $L_q^{II}$ above such a $p$.
We already saw that $\left(\frac{M/\QQ}{p}\right) = \sigma$.
Since the Artin symbol $\left(\frac{M/\QQ}{p}\right)$ is also given by $\rho|_M$, we must have 

\begin{equation}\label{eq:2.33}
\rho|_M = \sigma.
\end{equation}
Furthermore, since $q\mid p+1$, we must have that $\rho|_{\QQ(\zeta_q)}$ is given by the involution
\begin{equation}\label{eq:2.44}
\rho|_{\QQ(\zeta_q)} : \; \zeta_q \mapsto \zeta_q^{-1}.
\end{equation}
Lastly, we claim that $\rho$ must be of order 2.
To see this, we show that $p$ factors into primes of inertia degree 2 in $L_q^{II}$.
Consider first how $p$ factors in $M$.
In this case we know that $p\mathcal{O}_M = \mathfrak{p}_1 \mathfrak{p}_2 \cdot ... \cdot \mathfrak{p}_{m'}$, where each $\mathfrak{p}_i$ is of inertia degree 2 over $p$.
We now look how each of these primes factors further in $L_q^{II}$.
Let $\mathfrak{p}$ be one of the primes in the factorization above.
We claim that this prime splits completely in $L_q^{II}$.
Since $q\mid \mathcal{N}(\mathfrak{p}) = p^2 -1$, $\mathfrak{p}$ splits completely in $M(\zeta_q)$.
And so, in order for this prime to split completely in $L_q^{II}$ it is enough to show that 
\begin{equation}\label{eq:tpower1}
T^q = \frac{\tau(\alpha_M)}{\sigma(\tau(\alpha_M))} \bmod{\mathfrak{p}}
\end{equation} 
has a solution in $M$ for all $\tau\in\gal{M/\QQ}/\BRAK{\sigma}$.
Applying $\tau^{-1}$ to \eqref{eq:tpower1} and noting that $\sigma$ is in the center of $\gal{M/\QQ}$, we get that \eqref{eq:tpower1} is equivalent to the equation
\begin{equation}\label{eq:tpower2}
T^q = \frac{\alpha_M}{\sigma(\alpha_M)} \bmod{\tau^{-1}(\mathfrak{p})}.
\end{equation} 
Since $\sigma$ is the Frobenius element mod $\tau^{-1}(\mathfrak{p})$, we see that \eqref{eq:tpower2} is equivalent to
\begin{equation}\label{eq:tpower3}
T^q = \alpha_M^{p-1} \bmod{\tau^{-1}(\mathfrak{p})}.
\end{equation} 
Taking \eqref{eq:2.3} to the power of $p-1$, we see that \eqref{eq:tpower3} indeed has a solution (we can take $T = \left(\mathcal{N}_{K/M}(\beta)\right)^{p-1}$).
It follows that $p$ splits into primes of inertia degree 2 in $L_q^{II}$, which proves that $\rho$ has order 2.

We now prove that there is only a single element $\Tilde{\rho}\in\gal{L_q^{II}/\QQ}$ of order 2 which satisfies \eqref{eq:2.33} and \eqref{eq:2.44}.
Indeed, let $\Tilde{\rho}$ be as above.
Every element in $\gal{L_q^{II}/\QQ}$ is determined by its image of $\zeta_q$ and of the elements $(\tau(\alpha_M)/\sigma(\tau(\alpha_M)))^{1/q}$ with $\tau \in \gal{M/\QQ}/\BRAK{\sigma}$.
For $\Tilde{\rho}$ we have from \eqref{eq:2.44} that $\Tilde{\rho}(\zeta_q) = \zeta_q^{-1}$.
Furthermore, from \eqref{eq:2.33} we must have that $\Tilde{\rho}$ sends $(\tau(\alpha_M)/\sigma(\tau(\alpha_M)))^{1/q}$ to $(\sigma(\tau(\alpha_M))/\tau(\alpha_M))^{1/q}\zeta_q^t$ for some $t\in\FF_q$.
However, the fact that $\Tilde{\rho}$ has order 2 (and that $q$ is odd) implies $t=0$, since 
\begin{multline*}
\Tilde{\rho}\circ\Tilde{\rho} \left(\left(\frac{\tau(\alpha_M)}{\sigma(\tau(\alpha_M))}\right)^{1/q}\right)
=
\Tilde{\rho}\left(\left(\frac{\sigma(\tau(\alpha_M))}{\tau(\alpha_M)}\right)^{1/q}\right)
\Tilde{\rho}\left(\zeta_q^t\right)
=\\
\Tilde{\rho}\left(\left(\frac{\tau(\alpha_M)}{\sigma(\tau(\alpha_M))}\right)^{1/q}\right)^{-1}
\zeta_q^{-t}
=
\left(\frac{\tau(\alpha_M)}{\sigma(\tau(\alpha_M))}\right)^{1/q}
\zeta_q^{-2t}.
\end{multline*}

To summarize, we found out that if $p\in\PP_C(X)$ and $p$ has an obstruction of type II at $q$, then $\left(\frac{L_q^{II}/\QQ}{p}\right) = \Tilde{\rho}$.
And so, appliying \autoref{thm:2.1} and \autoref{lem:2.2} with the conjugacy class $\SET{\Tilde{\rho}}\subset \gal{L_q^{II}/\QQ}$ we get that
\begin{equation}\label{eq:2.8}
B_\alpha^{II}(X;q) \leq
\frac{\li{X}}{[L_q^{II}:\QQ]} + \BigO{\sqrt{X}\log X}.
\end{equation}

From our assumption \eqref{eq-assump_n/cn}, we know that $\alpha_M/\sigma(\alpha_M)$ is not a root of unity.
As a result we get that $[L_q^{II}:\QQ] \gg (q-1)q$.
And so, summing \eqref{eq:2.8} for $g(X)\leq q\leq \sqrt{X}/\log^2 X$ we get
$$
\sum_{g(X)\leq q \leq \frac{\sqrt{X}}{\log^2 X}}
B_\alpha^{II}(X;q)
\ll
\frac{\li{X}}{g(X)} + \frac{X}{\log^2 X}
= o(\li{X}).
$$
\end{proof}

\autoref{lem:2.4} takes care of obstructions at primes $g(X)\leq q \leq \frac{\sqrt{X}}{\log^2 X}$.
We now give two more lemmas, taking care of obstructions in the ranges
$\frac{\sqrt{X}}{\log^2 X}\leq q \leq \sqrt{X}\log X$ and $\sqrt{X}\log X \leq q \leq X$.

\begin{lemma}\label{lem:2.5}
$$
\sum_{\frac{\sqrt{X}}{\log^2 X}\leq q \leq \sqrt{X}\log X}
B_\alpha(X;q)
= \BigO{\frac{X\log \log X}{\log^2 X}}
$$
\end{lemma}

\begin{proof}
The proof of this case is identical to Hooley's original arguments. 
Let $p\in\PP_C(X)$ be a prime with an obstruction at $q$.
Since the range of $q$ is fairly small, we can ignore part of the information.
We will only use the fact that an obstruction at $q$ implied that $q\mid p^2 - 1$.
In other words, $p\equiv \pm 1\bmod{q}$.
Using \nameref{thm:2.3}, we get that for $\frac{\sqrt{X}}{\log^2 X}\leq q \leq \sqrt{X}\log X$, the number of such $p$'s is bounded by 
$$
B_\alpha(X;q) < \frac{4X}{(q-1)\log\left(X/q\right)}\ll \frac{\li{X}}{q}.
$$
Summing this inequality for all $\frac{\sqrt{X}}{\log^2 X}\leq q \leq \sqrt{X}\log X$ gives the result.
\end{proof}

\begin{lemma}\label{lem:2.6}
We have the bound:
$$
\sum_{\sqrt{X}\log X \leq q \leq X}
B^{II}_\alpha(X;q)
= \BigO{\frac{X}{\log^2 X}}.
$$
If we further assume $|\mathcal{N}_{K/\QQ}(\alpha)|\neq 1$
then we also have:
$$
\sum_{\sqrt{X}\log X \leq q \leq X}
B_\alpha(X;q)
= \BigO{\frac{X}{\log^2 X}}
$$
\end{lemma}

\begin{proof}
We write $\alpha = \alpha_1/\alpha_2$ with $\alpha_1,\alpha_2\in\mathcal{O}_K$.

Let $p\in\PP_C(X)$.
Assume first that  $|\mathcal{N}_{K/\QQ}(\alpha)|\neq 1$ and that $p$ has an obstruction of type I at some $q>\sqrt{X}\log X$.
If $\sigma\in C\subset\gal{K/\QQ}$ is the Artin symbol of one of the primes $\mathfrak{P}$ in $K$ above $p$, then
$$
\alpha^{(p^2 - 1)/q} \equiv
\left(\alpha \sigma(\alpha)\right)^{(p-1)/q}\equiv
1 \bmod{\mathfrak{P}}.
$$
Denoting $r = (p-1)/q$, we find that
$$
\mathfrak{P} \mid 
\left( \left(\alpha\sigma(\alpha)\right)^{r} - 1\right).
$$
It follows that
$$
p\mid \mathcal{N}_{K/\QQ}\left(\left(\alpha_1\sigma(\alpha_1)\right)^{r} - \left(\alpha_2\sigma(\alpha_2)\right)^{r}\right).
$$
And so, all primes $p\in\PP_C(X)$ which have a type I obstruction at some $q>\sqrt{X}\log X$ must divide
$$
P_I = 
\prod_{\sigma \in C}\prod_{r \leq \frac{\sqrt{X}}{\log X}}
\mathcal{N}_{K/\QQ}\left(\left(\alpha_1\sigma(\alpha_1)\right)^{r} - \left(\alpha_2\sigma(\alpha_2)\right)^{r}\right)
$$

We assumed that $|\mathcal{N}_{K/\QQ}(\alpha)|\neq 1$.
This implies that $\alpha\sigma(\alpha)$ is not a root of unity (for otherwise $\mathcal{N}_{K/\QQ}(\alpha)^2 = \mathcal{N}_{K/\QQ}(\alpha\sigma(\alpha)) = \pm 1$).
It follows that $P_I\neq 0$.

We now show that there is some $T > 0$ (depending on $\alpha,K$) such that 
$$
\mathcal{N}_{K/\QQ}\left(\left(\alpha_1\sigma(\alpha_1)\right)^{r} - \left(\alpha_2\sigma(\alpha_2)\right)^{r}\right) \ll T^r.
$$
Denote $a = \alpha_1\sigma(\alpha_1)$, $b =  \alpha_2\sigma(\alpha_2)$.
Let $\tau_1,...,\tau_{|G|}$ be the embeddings of $K$ into $\CC$.
Denote by $M$ the maximum over all $|\tau_i(a)|$ and $|\tau_i(b)|$.
Then
\begin{equation*}
\left|\mathcal{N}_{K/\QQ}\left(a^r - b^r\right)\right|
=\prod_{i=1}^{|G|}\left|\tau_i\left(a^r - b^r\right)\right|
\ll 
\prod_{i=1}^{|G|}\left(
\left|\tau_i(a)\right|^r + \left|\tau_i(b)\right|^r \right)
\leq \left(2M\right)^{|G|r},
\end{equation*}
so that we can choose $T = \left(2M\right)^{|G|}$.

It follows that 
$$
\log P_I \ll \sum_{r \leq \frac{\sqrt{X}}{\log X}} r
= \BigO{\frac{X}{\log^2 X}}
$$
where the implied constant depends on $K,\alpha, C$.

And so, since the number of primes dividing $P_I$ is bounded by $\log P_I$, we get that the number of primes $p\in \PP_C(X)$ having a type I obstruction at some $q\geq \sqrt{X}\log X$ is $\BigO{\frac{X}{\log^2 X}}$.

The case of obstructions of type II is similar.
If a prime $p\in\PP_C(X)$ has an obstruction of type II at some $q>\sqrt{X}\log X$, and if $\sigma\in C$ is the Artin of some prime $\mathfrak{P}$ in $K$ above $p$, then
$$
\alpha^{(p^2 - 1)/q} \equiv
\left(\frac{\sigma(\alpha)}{\alpha}\right)^{(p+1)/q}\equiv
1 \bmod{\mathfrak{P}}.
$$
It follows that all $p\in\PP_C(X)$ which have a type II obstruction at some $q>\sqrt{X}\log X$ must divide
$$
P_{II} = 
\prod_{\sigma \in C}\prod_{r \leq \frac{\sqrt{X}}{\log X}}
\mathcal{N}_{K/\QQ}\left(\left(\sigma(\alpha_1)\alpha_2\right)^{r} -\left(\sigma(\alpha_2)\alpha_1\right)^{r}\right).
$$
Our assumption that $\prod_{n\in N}n(\alpha) / \sigma\left(\prod_{n\in N}n(\alpha)\right)$ is not a root of unity implies that $\frac{\alpha}{\sigma(\alpha)}$ is also not a root of unity.
Thus, $P^{II}\neq 0$.
Once more, we find that 
$$
\log P_{II} = \BigO{\frac{X}{\log^2 X}}
$$
which implies that the number of primes $p\in \PP_C(X)$ having a type II obstruction at some $q\geq \sqrt{X}\log X$ is also $\BigO{\frac{X}{\log^2 X}}$.
This completes the proof.
\end{proof}

We now turn our attention to obstructions at small primes.
It is possible for $\alpha$ to have order smaller than $\frac{p^2-1}{f(X)}$ in $\left(\mathcal{O}_K/p\mathcal{O}_K\right)^\times$ without $p$ having an obstruction at any prime $q > f(X)$.
This can be the case if $\alpha$ has multiple obstructions at small primes $q < f(X)$.
We say that $\alpha$ has an obstruction of degree $\ell$ at $q$ if $q^\ell \mid (p^2 - 1)$ and $\alpha$ is a $q^\ell$ power mod $p\mathcal{O}_K$.
We will also say that $p$ has an obstruction at $n$ if it has an obstruction at each prime power dividing $n$.

\begin{lemma}\label{lem:2.7}
Let $f(x)$ be a function tending to infinity with $x$.
There exists a function $g(x)$ tending to infinity with $x$ such that the number of primes $p\in \PP_C(X)$ having an obstruction at some $n>f(X)$ which is $g(X)$ smooth is $o(\li{X})$.
\end{lemma}
\begin{proof}
Let $g(X)$ be some number (to be determined later).
Denote by $S$ the number of primes $p\in \PP_C(X)$ which have an obstruction at some $n>f(X)$ which is $g(X)$ smooth, and let $p$ be such a prime.
Since $p$ has an obstruction at all divisors of $n$, and since $n$ is $g(X)$ smooth, it follows that $p$ has an obstruction at some $f(X) \leq m < f(X)g(X)$ which is $g(X)$ smooth.
We can assume that $f(X)$ and $g(X)$ tend to infinity sufficiently slowly such that $f(X)g(X) < X^\theta$ for some $\theta < 1$.

Since $p$ has an obstruction at $m$, we must have $m\mid (p^2 - 1)$.
Using \nameref{thm:2.3}, we can bound the number of such $p$'s by
$$
 \#\SET{p\leq X \;\middle|\; m\mid p^2 - 1}
\ll \frac{Xd(m)}{\varphi(m) \log (X/m)}
$$
where $d(\cdot)$ is the divisor function.

It follows that the amount of $p\in \PP_C(X)$ having an obstruction at some $n > f(X)$ which is $g(X)$ smooth is bounded by
$$
S \ll \sum_{\substack{f(X) \leq  m < f(X)g(X) \\ m \text{ is }g(X)\text{ smooth}}}\frac{X d(m)}{\varphi(m) \log (X/m)}.
$$

Since $f(X)g(X)<X^\theta$, we have $\log(X/m) \gg \log(X)$, so that 
$$
S \ll
\frac{X}{\log X}
\sum_{\substack{f(X) \leq m < f(X)g(X) \\ m \text{ is }g(X)\text{ smooth}}}\frac{ d(m)}{\varphi(m)}.
$$
We can crudely bound this sum using Rankin's trick (see for example \cite{Ramaré2022}).
we have
\begin{multline*}
S 
\ll 
\frac{X}{\log X}\sum_{m \text{ is }g(X)\text{ smooth}}
\left(\frac{m}{{f(X)}}\right)^{1/2}\frac{d(m)}{\varphi(m)}
\ll \\
\frac{X}{(f(X))^{1/2}\log X}\prod_{p\leq g(X)}\left( 1 + \frac{p(2p^{1/2} - 1)}{(p-1)(p^{1/2} - 1)^2}\right).
\end{multline*}
We can choose $g(x)$ tending to infinity with $x$ sufficiently slowly in terms of $f(x)$ such that
$$
\frac{1}{(f(X))^{1/2}}\prod_{p\leq g(X)}\left(1 + \frac{p(2p^{1/2} - 1)}{(p-1)(p^{1/2} - 1)^2}\right) \xrightarrow{X\rightarrow \infty} 0.
$$
With such a choice we have $S = o(\li{X})$ as required.
\end{proof}

We are now ready to prove \autoref{thm:1.4}.
\begin{proof}
Let $f(x)$ be a function tending to infinity with $x$.
From \autoref{lem:2.7} we know that there is a function $g(x)$ tending to infinity with $x$ such that the number of primes $p\in \PP_C(X)$ having obstructions at some $n>f(X)$ which is $g(X)$ smooth is $o(\li{X})$.
In other words, the number of $p\in \PP_C(X)$ whose only prime obstructions are at primes $q \leq g(X)$ is  $o(\li{X})$.

If $|\mathcal{N}_{K/\QQ}(\alpha)|\neq 1$, then from \autoref{lem:2.4}, \autoref{lem:2.5}, \autoref{lem:2.6} we know that the number of  $p\in \PP_C(X)$ having obstructions at any $q > g(X)$ is $o(\li{X})$.
From this it follows that the number of $p\in \PP_C(X)$ with $\ord{p}{\alpha} \leq \frac{p^2 - 1}{f(X)}$ is $o(\li{X})$.

If $\left|\mathcal{N}_{\QQ/K}(\alpha)\right| = 1$, then from \autoref{lem:2.4}, \autoref{lem:2.5}, \autoref{lem:2.6} we know that the number of  $p\in \PP_C(X)$ having obstructions of type II at any $q > g(X)$ is $o(\li{X})$.
From this it follows that the number of $p\in \PP_C(X)$ with $\ord{p}{\alpha} \leq \frac{p+1}{f(X)}$ is $o(\li{X})$.
\end{proof}


\section{The Conjecture for Matrices }\label{sec:3}

In the previous sections we considered Artin's primitive root conjecture in the number field setting.
That is, we looked at a number field $K$ and some $\alpha\in K$, and we were interested in $\ord{p}{\alpha}$ for various rational primes $p$.

In this section we consider another variation of Artin's primitive root conjecture, that of rational matrices.
For a rational matrix $A\in\mathrm{GL}_n(\QQ)$ we can consider the order of $A\bmod p$ in $\mathrm{GL}_n(\FF_p)$ for various rational primes $p$.
We prove an analogue of \autoref{thm:1.4}:

\begin{theorem}\label{cor:3.2}
Let $A\in \mathrm{GL}_{n}(\QQ)$.
Denote by $f_A(t)\in\QQ[t]$ the characteristic polynomial of $A$.
Assume that $f_A(t)$ is irreducible.
Denote by $K$ the splitting field of $f_A$ over $\QQ$, and write $G = \gal{K/\QQ}$.
Let $C\subset G$ be some conjugacy class of elements of order 2.
Assume that there is a normal subgroup $N\triangleleft G$ such that $C\not\subset N$ and the image of $C$ in $G/N$ is in the center of $G/N$.
Assume also that for any eigenvalue $\alpha$ of $A$ and any $c\in C$,
$$
\left.\prod_{n\in N}n(\alpha) \middle/
c\left(\prod_{n\in N}n(\alpha) \right)
\right.
$$
is not a root of unity.
Then assuming GRH, for every function $f(x)$ tending to infinity with $x$, we have the following:

{\setlength{\parindent}{0cm}
If $\mathrm{det}(A)\neq \pm 1$ then
$$
\#\SET{p\in \PP_{C}(X) \;\middle| \; \ord{p}{A}\leq \frac{p^2 - 1}{f(X)}} = o(\li{X}).
$$
If $\mathrm{det}(A) = \pm1$ then
$$
\#\SET{p\in \PP_{C}(X) \;\middle| \; \ord{p}{A}\leq \frac{p + 1}{f(X)}} = o(\li{X}).
$$
}
\end{theorem}

\subsection{Equivalence of matrices and number fields}
Let $A\in \mathrm{GL}_{n}(\QQ)$.
We denote by $f_A(t)$ the characteristic polynomial of $A$.
For almost all rational primes $p$, the matrix $A\bmod{p}$ is in $\mathrm{GL}_{n}(\FF_p)$.
We can then consider $\ord{p}{A}$, the order of $A\bmod{p}$ in  $\mathrm{GL}_{n}(\FF_p)$.

If we assume that $f_A(t)$ is irreducible, we show that for almost all $p$, $\ord{p}{A}$ depends only on the polynomial $f_A(t)$.
As a result, we get that the problem of understanding $\ord{p}{A}$ is equivalent to a question about orders of algebraic numbers.
That is, we show that $\ord{p}{A}$ is equal to the order of any root $\alpha\in K^\times$ of $f_A(t)$ in $\left(\mathcal{O}_K/p\mathcal{O}_K\right)^\times$, where $K$ is the splitting field of $f_A(t)$.
We can then use the results regarding number fields and apply them for matrices as well.

We begin by proving the following theorem, relating the order of a matrix $A$ to the order of its eigenvalues:

\begin{theorem}\label{prop:mat}
Let $A\in \mathrm{GL}_{n}(\QQ)$.
Denote by $f_A(t)$ the characteristic polynomial of $A$.
Denote by $K$ the splitting field of $f_A(t)$, and let $\alpha_1,...,\alpha_n$ be the roots of $f_A(t)$ in $K$.
Then for all sufficiently large rational primes $p$ the following holds:

{\setlength{\parindent}{0cm}
If $A$ is diagonalizable over $K$ then
$$
\ord{p}{A} = \mathrm{LCM}\left( \SET{\ord{p}{\alpha_i}}_{i=1,2,...,n}\right).
$$
If $A$ is not diagonalizable over $K$ then 
$$
\ord{p}{A} = p\times \mathrm{LCM}\left( \SET{\ord{p}{\alpha_i}}_{i=1,2,...,n}\right).
$$
}
\end{theorem}

\begin{proof}
Let $J$ be the Jordan form of $A$, and write
\begin{equation}\label{eq:AJ}
A = U^{-1}JU.
\end{equation}
Assume further that $U$ is chosen such that $U\in M_n(\mathcal{O}_K)$.

We choose some natural number $d_A\in\NN$ such that $d_AA\in M_n(\ZZ)$. 
Similarly, we choose $d_J\in\NN$ such that $d_J J \in M_n(\mathcal{O}_K)$.

Let $p$ be a rational prime.
We will assume that $p$ is large enough so that $p$ does not divide $d_A, d_J$, and that $p\OK$ is prime to all $K$-primes dividing $\det(U)$ and all roots $\alpha_i$.

We denote 
$$
\ell = \mathrm{LCM}\left( \SET{\ord{p}{\alpha_i}}_{i=1,2,...,n}\right).
$$

We first prove that $\ell\mid \ord{p}{A}$, and that if $J$ is not diagonal then $p\ell \mid \ord{p}{A}$.
That is, if we denote $k = \ord{p}{A}$ then we wish to show that $\ord{p}{\alpha_i}\mid k$ for all $i\leq n$, and that if $J$ is not diagonal then $p\mid k$ as well.

By definition of $k$, we have that $A^k\equiv I \bmod p$.
That is, there is some matrix $B\in M_n(\ZZ)$ such that
\begin{equation*}
d_A^k A^k = d_A^k I + p B.
\end{equation*}
Using \eqref{eq:AJ} this implies that 
\begin{equation*}
d_A^kU^{-1}J^kU = d_A^K I + pB.
\end{equation*}
Thus, we get that
\begin{equation}\label{eq:Jk}
d_J^k J^k = d_J^k I + p \left(\frac{d_J}{d_A}\right)^kUBU^{-1}.
\end{equation}
Denote 
$$
C = \left(\frac{d_J}{d_A}\right)^kUBU^{-1}.
$$
On the one hand, we see that 
$$
C\in M_n\left(
\OK\left[\det(U)^{-1}, d_A^{-1}\right]
\right).
$$
On the other hand, since $C = \frac{1}{p}\left(d_J^k J^k - d_J^k I\right)$, we see that
$$
C\in M_n\left(\OK[p^{-1}]\right).
$$
Since $p$ is prime to $\det(U), d_A,$ it follows that $C$ is in $M_n(\OK)$.

Equation \eqref{eq:Jk} then implies that $J^k\equiv I \bmod p$.
In particular, if we consider the main diagonal of this equation, we see that $\alpha_i^k \equiv 1 \bmod p$ for all $i\leq n$.
This proves that $\ell\mid k$.

Assume now that $J$ is not a diagonal matrix.
This means that $J$ contains a block of the form
$$
J(\lambda) = 
\begin{pmatrix}
\lambda & 1 & 0 &\dots & 0 \\
0 & \lambda & 1 &\dots & 0 \\
\vdots & \vdots & \ddots & \vdots & \vdots\\
0 & 0 & \dots & \lambda & 1 \\
0 & 0 & \dots & 0 & \lambda
\end{pmatrix}.
$$
When looking at $J^k$, this block then has the form
$$
\begin{pmatrix}
\lambda^k & \binom{k}{1} \lambda^{k-1} & \binom{k}{2} \lambda^{k-2} &\dots \\
0 & \lambda^k & \binom{k}{1} \lambda^{k-1} &\dots \\
\vdots & \vdots & \ddots & \vdots\\
0 & \dots & 0 & \lambda^k
\end{pmatrix}.
$$
And so, in the equation $J^k\equiv I \bmod p$, we consider an element of the superdiagonal of this block.
We see that 
$$
\binom{k}{1}\lambda^{k-1} \equiv 0 \bmod p.
$$
Since $\lambda = \alpha_m$ for some $m$, and we assumed that $p$ is prime to all roots $\alpha_i$, it follows that we must have $p\mid \binom{k}{1} = k$.
Since $\ell \mid \left|\left(\OK/p\OK\right)^\times\right|$, it follows that $\ell$ is prime to $p$.
Thus, $p\mid k$ and $\ell \mid k$ imply $p\ell\mid k$ as required.

We now prove the converse.
Denote $t = \ell$ if $J$ is diagonal and $t = p\ell$ otherwise.
We wish to show that $\ord{p}{A}\mid t$.
That is, we want to show that $A^t\equiv I \bmod p$.

We first note that $t = \ord{p}{J}$, so that we have 
\begin{equation*}
d_J^t J^t \equiv  d_J^t I + p B
\end{equation*}
for some $B\in M_n(\OK)$.
From \eqref{eq:AJ} we get
\begin{equation*}
d_J^t U A^t U^{-1} = d_J^t I + p B,
\end{equation*}
which then gives
\begin{equation}\label{eq:At}
d_A^t A^t = d_A^t I + p \frac{d_A^t}{d_J^t}U^{-1}BU.
\end{equation}
Denoting $C = \frac{d_A^t}{d_J^t}U^{-1}BU$ we see that $C$ is in $M_n\left(\OK\left[ d_J^{-1}, \det(U)^{-1}\right]\right)$.
However, since $C = \frac{1}{p}\left(d_A^t A^t - d_A^t I\right)$, we have $C\in M_n\left(\ZZ[p^{-1}]\right)$ as well.
From our choice of $p$, it follows that $C$ is in $M_n(\ZZ)$.
Thus, \eqref{eq:At} implies that $A^t\equiv I \bmod p$, so that $\ord{p}{A}\mid t$ as required.

\end{proof}

This proposition now allows us to use \autoref{thm:1.4} and apply it to matrices.
As a result, we get \autoref{cor:3.2}:
\begin{proof}[Proof of \autoref{cor:3.2}] 
Let $A\in \mathrm{GL}_n(\QQ)$ be as in the statement of the theorem.
Since $f_A(t)$ is irreducible, it follows that $A$ is diagonalizable over $K$.
Thus, from \autoref{prop:mat} we get that for almost all primes $p$, 
$$
\ord{p}{A} = \mathrm{LCM}(\ord{p}{\alpha_1},...,\ord{p}{\alpha_n}),
$$
where $\alpha_i$ are the roots of $f_A$.
Since the Galois group $G$ acts transitively on these roots, we have that all of the $\ord{p}{\alpha_i}$ are equal.
And so, for almost all primes $p$ we have $\ord{p}{A} = \ord{p}{\alpha_1}$.

Applying \autoref{thm:1.4} to $\alpha_1$ then gives the required result about $A$.
\end{proof}

\bibliographystyle{plain}
\bibliography{my_bib}

\end{document}